\documentclass[11pt]{amsart}

\usepackage{xcolor, hyperref}
\usepackage{esint} 
\usepackage{amsmath}
\usepackage{verbatim} 

\newtheorem{theorem}{Theorem}[section]
\newtheorem{lemma}[theorem]{Lemma}
\newtheorem{proposition}[theorem]{Proposition}
\newtheorem{corollary}[theorem]{Corollary}

\newcommand{\B}{\mathbb{B}}

\newcommand{\x}{\mathbf{x}}

\renewcommand{\S}{\mathbb{S}}

\newtheorem{maintheorem}{Theorem}

\theoremstyle{remark} 
\newtheorem{remark}{Remark}

\numberwithin{equation}{section}

\usepackage{ulem}

\title[Rigidity of surfaces with nonpositive Euler characteristic]{Rigidity of surfaces with nonpositive Euler characteristic by the second eigenvalue of the Jacobi operator}

\author[Batista]{M. Batista}
\author[Cavalcante]{M.\,P.\,Cavalcante} 
\author[Mendes]{A. Mendes}
\author[Nunes]{I. Nunes}

\address{
  Instituto de Matem\'atica, Universidade Federal de Alagoas, Macei\'o - Brazil}
\email{mhbs@mat.ufal.br}
 \email{marcos@pos.mat.ufal.br}
 \email{abraao.mendes@im.ufal.br}
\address{
Departamento de Matemática, Universidade Federal do Maranhão,
 São Luís - Brazil}
\email{ivaldo.nunes@ufma.br}

 \subjclass[2020]{Primary 53C24, 58J50; Secondary 49Q05, 53A10.}

\date{\today}

\keywords{Jacobi operator, second eigenvalue, spectral rigidity, Willmore energy, minimal surfaces, stability index.}

\usepackage{marginnote}

\begin{document}

\begin{abstract}
In this paper, we investigate the spectral properties of the Jacobi operator for immersed surfaces with nonpositive Euler characteristic, extending previous results in the field. We first prove a sharp upper bound for the second eigenvalue of the Jacobi operator for compact surfaces with nonpositive Euler characteristic that are fully immersed in the Euclidean sphere, and then we classify all such surfaces attaining this upper bound. Furthermore, we demonstrate that totally geodesic tori maximize the second eigenvalue among all compact orientable surfaces with positive genus in the product space $\mathbb{S}^1(r) \times \mathbb{S}^2(s)$.
\end{abstract}

\maketitle
\section{Introduction}
The study of eigenvalues of differential operators on submanifolds has profound implications in both geometry and physics. In particular, the Jacobi operator, which represents the linearization of the mean curvature, plays a central role in understanding the stability of minimal submanifolds. 

For submanifolds $\Sigma$ immersed in Euclidean space, {the Jacobi} operator is expressed as 
\[
L = -\Delta - |\sigma|^2,
\]
where $\Delta$ is the Laplace operator, and $|\sigma|^2$ is the squared norm of the second fundamental form of $\Sigma$. 

In this setting, Harrell and Loss \cite{HarrellLoss} demonstrated that if $\Sigma^n$ is a compact orientable hypersurface in $\mathbb{R}^{n+1}$, then the second eigenvalue $\lambda_2(L)$ of $L$ is {nonpositive}, with equality ($\lambda_2(L) = 0$) occurring if and only if $\Sigma$ is a round sphere. This result resolved a conjecture previously posed by Alikakos and Fusco \cite{AlikakosFusco}.

El Soufi and Ilias, in their seminal work \cite{ElSoufiIlias}, extended the results of Harrell and Loss by providing sharp upper bounds for the second eigenvalue of the Schr\"odinger operator $S = -\Delta + q$, for compact submanifolds in simply connected space forms, in terms of the total mean curvature of $\Sigma$ and the mean value of the potential $q$. Their results further demonstrated that for compact submanifolds in these spaces, the second eigenvalue of the Jacobi operator is {nonpositive}, achieving zero if and only {if} the submanifold is a geodesic sphere (see \cite[Corollary 2.1]{ElSoufiIlias}).

Building on these foundational results, the third-named author \cite{Mendes} further explored the spectral properties of the Jacobi operator for immersed hypersurfaces, providing new geometric insights into eigenvalue bounds. Specifically, he characterized certain closed immersed hypersurfaces through the second eigenvalue of the Jacobi operator, extending the work of El Soufi and Ilias to more general contexts, including warped product manifolds. Additionally, he showed that the Clifford tori in $\mathbb{S}^3$ maximize the second eigenvalue of the Jacobi operator among all compact orientable surfaces with positive genus.

In this paper, we extend previous results by characterizing surfaces immersed in higher-dimensional spheres through the second eigenvalue of the Jacobi operator. More precisely, our first theorem is as follows:

\begin{maintheorem}\label{theorem1}
Let \(\Sigma\) be a closed surface {with nonpositive Euler characteristic}, fully immersed in \(\mathbb{S}^n\). Then, the second eigenvalue of the Jacobi operator of \(\Sigma\) in \(\mathbb{S}^n\),
\[
L = -\Delta - |\sigma|^2 - 2,
\]
satisfies
\[
\lambda_2(L) \leq -2.
\]
Furthermore, if \(\lambda_2(L) = -2\), then \(\Sigma\) is orientable and one of the following holds:
\begin{enumerate}
    \item \(n = 3\) and \(\Sigma\) is congruent to the Clifford torus;
    \item \(n = 5\) and \(\Sigma\) is congruent to the equilateral torus.
   
\end{enumerate}
\end{maintheorem}

In fact, we obtain a more general result that estimates  \(\lambda_2(L)|\Sigma|\) in terms of the Willmore energy (see Theorem \ref{theoremW}).

Additionally, we characterize totally geodesic tori using the second eigenvalue of the Jacobi operator in the product space \(\mathbb{S}^1(r) \times \mathbb{S}^2(s)\). More precisely:

\begin{maintheorem}\label{theorem2}
Let \(\Sigma\) be a closed, orientable surface of positive genus immersed in \(M = \mathbb{S}^1(r) \times \mathbb{S}^2(s)\). If \(r \geq s\), then the second eigenvalue \(\lambda_2(L)\) of the Jacobi operator of \(\Sigma\) in \(M\),
\[
L = -\Delta - (|\sigma|^2 + \operatorname{Ric}_M(N,N)),
\]
satisfies
\[
\lambda_2(L) \leq 0.
\]
Furthermore, if \(\lambda_2(L) = 0\), then \(r = s\), and \(\Sigma\) is congruent to the totally geodesic torus \(\mathbb{S}^1(r) \times \mathbb{S}^1(r)\) in \(\mathbb{S}^1(r) \times \mathbb{S}^2(r)\).
\end{maintheorem} 

\begin{remark}
In \cite{Urbano2}, Urbano proved that the totally geodesic embedding $\mathbb{S}^1(r)\times \mathbb{S}^1(r)\subset\mathbb{S}^1(r)\times \mathbb{S}^2(r)$ is the only compact orientable minimal surface immersed in $\mathbb{S}^1(r)\times \mathbb{S}^2(r)$ with index one. It is worth to note that Theorem \ref{theorem2} provides an alternative proof of that result. In fact, let $\Sigma$ be an orientable compact minimal surface immersed in $\mathbb{S}^1(r)\times \mathbb{S}^2(r)$ with index one. First observe that $\Sigma$ has positive genus because, if not, $\Sigma$ should be equal to a slice $\{p\}\times\mathbb{S}^2(r)$, for some $p\in\mathbb{S}^1(r)$ (see \cite{ManzanoPlehnertTorralbo}), which are known to be stable. Thus, on one hand, by Theorem \ref{theorem2}, we have $\lambda_2(L)\leq 0$ and, on the other hand, $\lambda_2(L)\geq 0$ since $\Sigma$ has index one. Thus, $\lambda_2(L)=0$ and the conclusion follows from the rigidity part in Theorem \ref{theorem2}.
\end{remark}

{\bf Outline of the paper.} 
In Section \ref{pre}, we develop some preliminaries results on the conformal area of closed surfaces in higher-dimensional spheres and discuss properties of the extrinsic geometry of \( \mathbb{S}^1(r) \times \mathbb{S}^n(\sqrt{1 - r^2}) \) in the sphere. 
In Section \ref{spheres}, we prove a result that implies Theorem \ref{theorem1}, using test functions obtained from conformal maps.  
In Section~\ref{torus}, we present the proof of Theorem~\ref{theorem2}.
Finally, in Section \ref{exa}, we conclude by presenting examples that demonstrate the necessity of the hypotheses in Theorem \ref{theorem2}.

\section{Preliminaries}\label{pre}

\subsection{Area estimates for surfaces in the unit sphere}

In this section, we establish an important estimate for the area of conformally transformed surfaces in the unit sphere \(\mathbb{S}^n\) (see \cite{MontielRos}). These estimates will be crucial in the proof of Theorem \ref{theorem1}, particularly in the analysis of the second eigenvalue of the Jacobi operator. To this end, we first recall some properties of conformal transformations in \(\mathbb{S}^n\).

Let \(\mathbb{S}^n\) be the unit sphere, and let \(\mathbb{B}^{n+1}\) denote the open unit ball in the \((n+1)\)-dimensional Euclidean space \(\mathbb{R}^{n+1}\), with \(n \geq 3\). For any \(y \in \mathbb{B}^{n+1}\), define the conformal map \(F_y: \mathbb{S}^n \to \mathbb{S}^n\) by
\begin{align*}
F_y(x) = \frac{1 - |y|^2}{|x + y|^2}(x + y) + y.
\end{align*}
A direct computation shows that 
\begin{align*}
F_y^*ds^2 = \rho_y^2\,ds^2,
\end{align*}
where \(\rho_y: \mathbb{S}^n \to \mathbb{R}\) is the positive function given by
\begin{align*}
\rho_y(x) = \frac{1 - |y|^2}{|x + y|^2}.
\end{align*}
Thus, \(F_y\) is a conformal diffeomorphism of \(\mathbb{S}^n\) for every \(y \in \mathbb{B}^{n+1}\), with \(F_0 = \operatorname{id}\).


Now, let \(\Sigma \subset \mathbb{S}^n\) be a closed, connected surface immersed in \(\mathbb{S}^n\). The area of its conformal image \(\Sigma_y = F_y(\Sigma) \subset \mathbb{S}^n\) is given by
\begin{align}\label{eq:aux1}
|\Sigma_y| = \int_\Sigma \left( \frac{1 - |y|^2}{|\x(p) + y|^2} \right)^2 dv(p),
\end{align}
where \(\x: \Sigma \to \mathbb{S}^n \subset \mathbb{R}^{n+1}\) is the position map, \(\x(p) = p\), and \(dv\) is the volume element (or the Riemannian density, see \cite[Appendix B]{L18} for more details) of \(\Sigma\) with respect to the metric induced from \(\mathbb{S}^n\).

To obtain an upper bound for \(|\Sigma_y|\), we introduce a key auxiliary function. Fix \(z \in \mathbb{B}^{n+1}\) and define \(f: \Sigma \to \mathbb{R}\) by \(f = 1 + \langle \x, z \rangle\). It is well known that
\begin{align*}
\Delta \x + 2\x = 2\vec{H},
\end{align*}
where \(\Delta\) is the Laplace operator on \(\Sigma\) and \(\vec{H}\) is the mean curvature vector of \(\Sigma\) in \(\mathbb{S}^n\). Consequently,
\begin{align*}
\Delta f = -2\langle \x, z \rangle + 2\langle \vec{H}, z \rangle,
\end{align*}
which implies
\begin{align}\label{eq:aux2}
\Delta \ln f = \frac{\Delta f}{f} - \frac{|\nabla f|^2}{f^2} = -1 + \frac{1 - |z|^2}{f^2} + \frac{|z^N|^2}{f^2} + \frac{2\langle \vec{H}, z \rangle}{f},
\end{align}
where \(z^N\) is the component of \(z\) normal to \(\Sigma\) and tangent to \(\mathbb{S}^n\).

With these preliminaries, we now establish the following key estimate relating \(|\Sigma_y|\) to the Willmore energy of \(\Sigma\).

\begin{lemma}\label{lem:area.conformal}
Let \(\Sigma\) be an immersed surface in \(\mathbb{S}^n\). Then, for every \(y \in \mathbb{B}^{n+1}\), we have
\[
|\Sigma_y| \leq \mathcal{W}(\Sigma),
\]
where the Willmore energy of \(\Sigma\) is given by
\begin{equation}\label{Willmore}
\mathcal{W}(\Sigma) = \int_\Sigma \big(1+ |\vec{H}|^2\big)\, dv.
\end{equation}
Moreover, if equality holds for some \(y\), then \(\vec{H}=-\frac{y^N}{f}\). In particular, if \(\Sigma\) is minimal and equality holds, then \(\Sigma\) is isometric to the round sphere \(\mathbb{S}^2\) or \(y = 0\).
\end{lemma}

\begin{proof}
Given $y \in \B^{n+1}$, define $z = 2(1 + |y|^2)^{-1}y \in \B^{n+1}$. Observe that
\begin{align}\label{eq:aux3}
\frac{1 - |z|^2}{(1 + \langle \x, z \rangle)^2} = \left( \frac{1 - |y|^2}{|\x + y|^2} \right)^2.
\end{align}
Substituting \eqref{eq:aux2} and \eqref{eq:aux3} into \eqref{eq:aux1}, using the divergence theorem and rearranging, we find that
\begin{align*}
|\Sigma| = |\Sigma_y| + \int_\Sigma \left(\left|\frac{z^N}{f}+\vec{H}\right|^2 - |\vec{H}|^2\right) dv,
\end{align*}
and from that we get the first part. 
In the nonorientable case, we use the orientable double covering $\pi: \hat{\Sigma} \to \Sigma$, equipped with the pullback metric. Consequently, the Riemannian density $d\hat{v}$ on $\hat{\Sigma}$ satisfies $\pi^*(\Delta h\, dv) = \hat{\Delta} \hat{h}\, d\hat{v}$. In particular, we obtain  
\[
\int_\Sigma \Delta h\, dv = 2\int_{\hat{\Sigma}} \hat{\Delta} \hat{h}\, d\hat{v} = 0,
\]
since $\pi$ is a local isometry and $\hat{h} = \pi^*(h)$.

If \(\Sigma\) is minimal and \(|\Sigma| = |\Sigma_y|\), then \(z^N = 0\), which is equivalent to \(y^N = 0\). Let \(h = \langle \x, y \rangle\) on \(\Sigma\). Then
\[
\nabla h = y - h\x \quad \text{and} \quad \nabla^2 h = -h \langle\cdot,\cdot\rangle.
\]
By Obata's theorem (cf.\ \cite[Theorem A]{O62}), either \(\Sigma\) is isometric to \(\mathbb{S}^2\), or \(h = 0\) on \(\Sigma\).
If \(h = 0\), then \(\nabla h = 0= y\).
\end{proof}

\subsection{Extrinsic geometry of \(\mathbb{S}^1(r) \times \mathbb{S}^n(\sqrt{1 - r^2})\) in \(\mathbb{S}^{n+2}\)}

In this section, we study the geometric properties of the product space \text{\(\mathbb{S}^1(r) \times \mathbb{S}^n(\sqrt{1 - r^2})\)} when viewed as a hypersurface of \text{\(\mathbb{S}^{n+2}\)}. These properties will be applied in the proof of Theorem~\ref{theorem2}.

Fix \( p \in M = \mathbb{S}^1(r) \times \mathbb{S}^n(\sqrt{1 - r^2}) \), and let \( \{e_1, e_2, \dots, e_{n+1}\} \) be an orthonormal basis of \( T_p M \), where \( e_1 \) is tangent to \( \mathbb{S}^1(r) \) and \( e_2, \dots, e_{n+1} \) are tangent to \( \mathbb{S}^n(\sqrt{1 - r^2}) \). A straightforward computation shows that \( e_1, e_2, \dots, e_{n+1} \) are principal directions of \( M \), with corresponding principal curvatures
\[
\kappa_1 = -\frac{\sqrt{1 - r^2}}{r}, \quad \kappa_2 = \dots = \kappa_{n+1} = \frac{r}{\sqrt{1 - r^2}}.
\]

The Weingarten operator of \( M \) in \( \mathbb{S}^{n+2} \) is given by
\begin{equation}\label{eqAX}
A(X) = -\dfrac{\sqrt{1 - r^2}}{r} \langle X, e_1\rangle e_1 + \sum_{i=2}^{n+1} \dfrac{r}{\sqrt{1 - r^2}} \langle X, e_i\rangle e_i, \quad X \in T_p M.
\end{equation}
Then we obtain
\begin{equation}\label{eqA2}
|A(X)|^2 = \left(\dfrac{1}{r^2} - \dfrac{1}{1 - r^2}\right)\langle X, e_1\rangle^2 + \dfrac{r^2}{1 - r^2}|X|^2.
\end{equation}

Let \( \Sigma \) be an immersed surface in \( M = \mathbb{S}^1(r) \times \mathbb{S}^n(\sqrt{1 - r^2}) \subset \mathbb{S}^{n+2} \). Denote the second fundamental forms of \( \Sigma \) in \( M \) and in \( \mathbb{S}^{n+2} \) by \( \sigma \) and \( \tau \), respectively. Then, we can write:
\begin{equation}\label{eq:taurhosigma}
\tau(X,Y) = \rho(X,Y) + \sigma(X,Y),
\end{equation}
and
\[
|\tau(X,Y)|^2 = |\rho(X,Y)|^2 + |\sigma(X,Y)|^2
\]
for all \( X, Y \in \mathfrak{X}(\Sigma) \), where \( \rho \) is the second fundamental form of \( M \) in \( \mathbb{S}^{n+2} \), and in scalar form is given by \( \rho(X,Y) = \langle A X, Y \rangle \).

Therefore, if \( \{X_1, X_2\} \) is an orthonormal basis for \( T_p \Sigma \), with \( p \in \Sigma \), then
\begin{align*}
|\tau|^2 &= \sum_{i,j=1}^2 |\tau(X_i, X_j)|^2 \\
&= \sum_{i,j=1}^2 |\rho(X_i, X_j)|^2 + \sum_{i,j=1}^2 |\sigma(X_i, X_j)|^2 \\
&= |\rho|^2_\Sigma + |\sigma|^2,
\end{align*}
where
\[
|\rho|^2_\Sigma := \sum_{i,j=1}^2 |\rho(X_i, X_j)|^2
\]
is independent of the orthonormal basis \( \{X_1, X_2\} \) chosen.

Moreover, note that
\begin{equation}\label{rhoSigma2}
|\rho|^2_\Sigma = \sum_{i,j=1}^2 |\rho(X_i, X_j)|^2 = |A(X_1)|^2 + |A(X_2)|^2 - \sum_{i=1}^2 |A(X_i)^\perp|^2,
\end{equation}
where \( A(X_i)^\perp \) denotes the orthogonal projection of \( A(X_i) \) onto \( (T_p \Sigma)^\perp \), the orthogonal complement of \( T_p \Sigma \) in \( T_p M \).

\begin{proposition}\label{prop:bound.psff}
Let \( \Sigma \) be a connected, orientable, closed surface immersed in \( \mathbb{S}^1(r) \times \mathbb{S}^n(\sqrt{1 - r^2}) \), with \( \frac{1}{\sqrt{2}} \leq r < 1 \). Then the following holds:
\begin{equation}\label{eqrho3}
|\rho|^2_\Sigma \leq \dfrac{2r^2}{1 - r^2}.
\end{equation}
Moreover, in the case of equality:
\begin{itemize}
\item[(i)] If \( r > \frac{1}{\sqrt{2}} \), then \( \Sigma \subset \{\theta\} \times \mathbb{S}^n(\sqrt{1 - r^2}) \) for some \( \theta \in \mathbb{S}^1(r) \);
\item[(ii)] If \( r = \frac{1}{\sqrt{2}} \), then either \( \Sigma \subset \{\theta\} \times \mathbb{S}^n(\sqrt{1 - r^2}) \) for some \( \theta \in \mathbb{S}^1(r) \), or \( \Sigma \) is a torus and \( 2K = 4|\vec{H}|^2 - |\sigma|^2 \) on \( \Sigma \).
\end{itemize}
\end{proposition}

\begin{proof}
Let \( p \in \Sigma \). Given any orthonormal basis \( \{X_1, X_2\} \) of \( T_p\Sigma \), it follows from equation \eqref{rhoSigma2} that
\[
|\rho|_\Sigma^2 \leq |A(X_1)|^2 + |A(X_2)|^2,
\]
with equality if and only if \( T_p\Sigma \) is invariant under the shape operator \( A: T_p M \to T_p M \).

From equation \eqref{eqA2}, we obtain:
\[
|A(X_1)|^2 + |A(X_2)|^2 = \left( \frac{1}{r^2} - \frac{1}{1 - r^2} \right)\left( \langle X_1, e_1 \rangle^2 + \langle X_2, e_1 \rangle^2 \right) + \frac{2r^2}{1 - r^2} \leq \frac{2r^2}{1 - r^2},
\]
where \( e_1 = \partial_\theta / r \). This proves inequality \eqref{eqrho3}.

Now suppose that equality holds in \eqref{eqrho3}. Then \( T_p\Sigma \) is invariant under \( A \) for all \( p \in \Sigma \). Since \( e_1 \) is a principal direction of \( A \) associated with the principal curvature \( -\sqrt{1 - r^2}/r \), we must have, at each point \( p \in \Sigma \),
\begin{itemize}
\item[(i)] either \( e_1(p) \perp T_p\Sigma \), or
\item[(ii)] \( e_1(p) \in T_p\Sigma \).
\end{itemize}

Note that case (i) implies \( \Sigma \subset \{\theta\} \times \mathbb{S}^n(\sqrt{1 - r^2}) \) for some \( \theta \in \mathbb{S}^1(r) \), which is the only possibility when \( r > 1/\sqrt{2} \).

Now suppose \( r = 1/\sqrt{2} \) and that \( e_1(p) \in T_p\Sigma \) for all \( p \in \Sigma \). Then for any orthonormal basis \( \{X_1, X_2\} \) of \( T_p \Sigma \) with \( X_1 = e_1 \), we have \( A(X_2) = X_2 \).

Applying the Gauss equation for the immersion \( \Sigma \subset M \) and using that the sectional curvature of \( M \) vanishes in the direction \( (X_1, X_2) \), that is, \( K_M(X_1, X_2) = 0 \), we obtain:
\begin{align*}
K &= \langle \sigma(X_1, X_1), \sigma(X_2, X_2) \rangle - |\sigma(X_1, X_2)|^2 \\
  &= \frac{1}{2}(4|\vec{H}|^2 - |\sigma|^2).
\end{align*}

Finally, since \( e_1(p) \in T_p\Sigma \) for all \( p \in \Sigma \), the restriction \( \pi_1|_\Sigma : \Sigma \to \mathbb{S}^1 \), where \( \pi_1 : \mathbb{S}^1(r) \times \mathbb{S}^n(\sqrt{1 - r^2}) \to \mathbb{S}^1(r) \) is the canonical projection \( \pi_1(\theta, w) = \theta \), is a submersion. This implies that \( \Sigma \) is foliated by circles, and therefore \( \Sigma \) is a torus.
\end{proof}

\section{Spectral characterization of surfaces with nonpositive Euler characteristic in higher-dimensional spheres}\label{spheres}

In this section, we discuss the spectral characterization of surfaces with nonpositive Euler characteristic immersed in higher-dimensional spheres. Such characterizations naturally arise in the study of extremal metrics and minimal immersions into spheres. This is, in fact, a highly active area of research; see, for instance, \cite{ElSoufiIlias2, ElSoufiIlias, Karpukhin2016, Karpukhin2019, KPS2025, Korevaar, LiYau, MontielRos, Nadirashvili, YangYau} and the references therein.

We begin by recalling some fundamental properties of the  special surfaces appearing in Theorem \ref{theorem1}. 

\subsection{Minimal {tori} in spheres}\label{tori}
The \textit{Clifford torus} is the standard immersion of the flat torus \( \mathbb{S}^1(1/\sqrt{2}) \times \mathbb{S}^1(1/\sqrt{2}) \) into \( \mathbb{S}^3 \). This surface is a particularly remarkable minimal surface in the 3-sphere and has been characterized in various ways. For example, it is characterized by the norm of its second fundamental form, as shown in \cite{CdCK}, and by its Morse index in \cite{Urbano1}. Moreover, it is known to minimize the Willmore energy among positive genus surfaces, as established in \cite{MN}, and it has been proven to be the unique genus one embedded minimal surface in \(\mathbb{S}^3\), according to \cite{B}.

The \textit{equilateral torus} is defined as the quotient \( \mathbb{R}^2 / \Gamma \), where
\[
\Gamma = \mathbb{Z}(1,0) \oplus \mathbb{Z}(1/2, \sqrt{3}/2).
\]

In \cite{Nadirashvili}, Nadirashvili showed that the flat metric corresponding to the equilateral lattice maximizes the first eigenvalue of the Laplace operator among all metrics of the same volume on the 2-torus.

This torus also admits an isometric immersion in the unit sphere \( \mathbb{S}^5 \subset \mathbb{R}^6 \) and, along with the Clifford torus, was also characterized by El Soufi and Ilias in the following theorem. 
We recall that an immersion \( h: \Sigma \to \mathbb{S}^{n} \) is said to be \textit{full} if its image is not contained in a great sphere of $\mathbb{S}^n$, that is, if the coordinate functions \( h_1, \dots, h_{n+1} \) are linearly independent.

\begin{theorem}[El Soufi-Ilias, \cite{ElSoufiIlias2}]\label{EI}
Let \( \Sigma \) be a topological torus fully immersed in the unit sphere \( \mathbb{S}^n \). If the coordinate functions of \( \Sigma \) are first eigenfunctions of the Laplace operator on \( \Sigma \), then either:
\begin{enumerate}
    \item[(a)] \( n = 3 \) and \( \Sigma \) is congruent to the Clifford torus; or
    \item[(b)] \( n = 5 \) and \( \Sigma \) is congruent to the equilateral torus.
\end{enumerate}
\end{theorem}

\subsection{The Klein bottle as a minimal surface in the sphere}\label{Klein}

The Klein bottle, \( K = \mathbb{R}P^2 \# \mathbb{R}P^2 \), is a closed, nonorientable surface of zero Euler characteristic.

We recall that, in his seminal work \cite{Lawson}, Lawson constructed infinite families of minimal surfaces in spheres for every genus. For instance, the map  
\[
\psi_{3,1}(x,y) = (\cos 3x \cos y, \sin 3x \cos y, \cos x \sin y, \sin x \sin y)
\]
defines a minimal immersion of \( \mathbb{R}^2 \) into \( \mathbb{S}^4 \), whose image is a torus denoted by \( \tau_{3,1} \).

Lawson observed that if \( \psi: S^2 \to \mathbb{S}^3 \) is a minimal immersion and \( \psi^*: S \to \mathbb{S}^5 \) denotes its Gauss map, then the bipolar map  
\[
\tilde\psi = \psi \wedge \psi^\ast : S \to \mathbb{S}^5
\]
is also a minimal immersion. It turns out that, in the case of \( \psi_{3,1} \) (and for other tori), its image, denoted by  \( \tilde\tau_{3,1} \), actually lies in \( \mathbb{S}^4 \), an equator of \( \mathbb{S}^5 \). By \cite[Theorem 1.3.1(3)]{Lapointe}, we know that \( \tilde\tau_{3,1} \) is a Klein bottle.

We now recall the following important spectral characterization of the Klein bottle as a minimal surface in \( \mathbb{S}^4 \) analogous to Theorem \ref{EI}. This result completes the previous work of Jakobson, Nadirashvili, and Polterovich in \cite{JakobsonNadirashviliPolterovich}.

\begin{theorem}[El Soufi-Giacomini-Jazar, \cite{EGJ06}]\label{EGJ}
Let \( \Sigma \) be a topological Klein bottle that is minimally and fully immersed in \( \mathbb{S}^n \). If the coordinate functions of \( \Sigma \) are first eigenfunctions of the Laplace operator on \( \Sigma \), then \( n = 4 \) and \( \Sigma \) is congruent to Lawson's bipolar surface \( \tilde{\tau}_{3,1} \).
\end{theorem}

\begin{remark}\label{nonflat}
    We point out that this metric is not flat, as proved by 
Nadirashvili \cite{Nadirashvili}. 

\end{remark}

\subsection{Eigenvalue estimate via Willmore {energy}}

In this subsection, we derive an upper bound for the second eigenvalue of the Jacobi operator of a surface in the sphere in terms of its Willmore energy \eqref{Willmore}. As a consequence, we obtain Theorem \ref{theorem1}. The result is stated as follows:

\begin{theorem}\label{theoremW}
Let \( \Sigma \) be a closed surface fully immersed in \( \mathbb{S}^n \). Then the second eigenvalue of the Jacobi operator
\[
L = -\Delta - |\sigma|^2 - 2
\]
satisfies the inequality
\[
\lambda_2(L)|\Sigma| \leq -2\mathcal{W}(\Sigma) + 4\pi \chi(\Sigma),
\]
where \( \mathcal{W}(\Sigma) \) denotes the Willmore energy of \( \Sigma \).

In particular, if \( \chi(\Sigma) \leq 0 \), then \( \lambda_2(L) \leq -2 \). If equality holds, that is, \( \lambda_2(L) = -2 \), then \(\Sigma\)  is orientable and one of the following occurs:
\begin{enumerate}
    \item \( n = 3 \) and \( \Sigma \) is congruent to the Clifford torus;
    \item \( n = 5 \) and \( \Sigma \) is congruent to the equilateral torus.
\end{enumerate}
\end{theorem}

\begin{proof}
To analyze the second eigenvalue of the Jacobi operator \( L \), we consider a first eigenfunction \(\varphi > 0\) of \(L\). Following an argument contained in the work of Li and Yau \cite{LiYau} (see also \cite{Hersch}), there exists \( y \in \mathbb{B}^{n+1} \) such that  
\[
\int_\Sigma \varphi(F_y \circ \x) \, dv = 0,
\]
or equivalently,
\[
\int_\Sigma \varphi \psi_i \, dv = 0, \quad i = 1, \dots, n+1,
\]
where \(\psi_1, \dots, \psi_{n+1}\) are the coordinate functions of \( \psi = F_y \circ \x \). Using these functions as test functions for \(\lambda_2 = \lambda_2(L)\), we obtain
\begin{align*}
\lambda_2 \int_\Sigma \psi_i^2 \, dv &\leq \int_\Sigma \psi_i L \psi_i \, dv \\
&= \int_\Sigma |\nabla \psi_i|^2 \, dv - \int_\Sigma (|\sigma|^2 + 2) \psi_i^2 \, dv, \quad i = 1, \dots, n+1.
\end{align*}
Summing over \( i = 1, \dots, n+1 \) and using \(\sum_i \psi_i^2 = 1\), we get
\[
\lambda_2 |\Sigma| \leq \int_\Sigma |\nabla \psi|^2 \, dv - \int_\Sigma |\sigma|^2 \, dv - 2 |\Sigma|,
\]
where \( |\nabla \psi|^2 = \sum_i |\nabla \psi_i|^2 \). 
It is well known that the Dirichlet energy of \( \psi = F_y \circ \x \) satisfies (see, e.g., \cite[p.\ 7]{EI84})
\[
\int_\Sigma |\nabla \psi|^2 \, dv = 2 |\Sigma_y|.
\]
Thus, we obtain
\[
\lambda_2 |\Sigma| \leq 2 |\Sigma_y| - \int_\Sigma |\sigma|^2 \, dv - 2 |\Sigma|.
\]

From Lemma \ref{lem:area.conformal}, we have the area bound
\[
|\Sigma_y| \leq \int_\Sigma (1 + |\vec{H}|^2) \, dv,
\]
with equality if and only if \( \vec{H}=-\frac{z^N}{f} \), where \( z = 2(1 + |y|^2)^{-1}y \) and \( f(x)=1+\langle x, z\rangle \). 

From the Gauss equation, we recall that
\[
|\sigma|^2 = 2 + 4 |\vec{H}|^2 - 2K,
\]
where \( K \) is the Gaussian curvature of \( \Sigma \). Substituting this identity and using the Gauss-Bonnet theorem (cf.\ \cite[Theorem 9.7]{L18}), we obtain
\begin{align*}
\lambda_2 |\Sigma| 
&\leq 2 |\Sigma_y| - \int_\Sigma |\sigma|^2 \, dv - 2 |\Sigma| \\
&\leq 2 \int_\Sigma (1 + |\vec{H}|^2) \, dv - \int_\Sigma (2 + 4 |\vec{H}|^2 - 2K) \, dv - 2 |\Sigma| \\
&= -2 \int_\Sigma |\vec{H}|^2 \, dv + 4\pi \chi(\Sigma) - 2 |\Sigma|,
\end{align*}
where \( \chi(\Sigma) \) is the Euler characteristic of \( \Sigma \). 
And, thus, if $\chi(\Sigma)\le0$, we have
\[
\lambda_2|\Sigma| \leq -2\mathcal W(\Sigma)\leq -2|\Sigma|.
\]

Now, if \( \lambda_2(L) = -2 \) then \( \vec{H} \) must vanish, and all the above inequalities must be equalities. Consequently:
\begin{enumerate}
    \item \( \chi(\Sigma) = 0 \), so \( \Sigma \) is either a torus or a Klein bottle;
    \item \( |\Sigma_y| = |\Sigma| \), implying \( y = 0 \) by Lemma \ref{lem:area.conformal};
    \item The coordinate functions \( \psi_1, \dots, \psi_{n+1} \) are eigenfunctions of \( L \) associated with \( \lambda_2 = -2 \), meaning
    \[
    \Delta \psi + |\sigma|^2 \psi = 0.
    \]
\end{enumerate}

Since \( y=0 \), we have \( \psi = \x \), so
\[
\Delta \x + |\sigma|^2 \x = 0.
\]
On the other hand, by the minimality of \( \Sigma \), we also have \( \Delta \x + 2 \x = 0 \), implying \( |\sigma|^2 = 2 \). In particular, by the Gauss equation, \(\Sigma\) is flat and \(-\Delta = L+4\). Since \(\lambda_2(L)=-2\) we get that the first non-zero eigenvalue of the Laplacian of \(\Sigma\) is \(2\).

Therefore we have that \(\Sigma\) is flat and minimally immersed by  the first eigenfunctions of the Laplacian. Note that it follows from Theorem \ref{EGJ} and Remark \ref{nonflat} that it cannot be a Klein bottle.

Thus, $\Sigma$ is a torus and we conclude by Theorem \ref{EI} that $\Sigma$ is congruent to one of the following:

\begin{itemize}
    \item the Clifford torus if \( n = 3 \);
    
    \item the equilateral torus if \( n = 5 \).
\end{itemize}
\end{proof}

Using the solution of the Willmore conjecture of Marques and Neves \cite[Theorem A]{MN}, we obtain the following  direct consequence:

\begin{corollary}
Let \(\Sigma\) be a closed, orientable surface of positive genus immersed in \(\mathbb{S}^3\). Then,  
\[
\lambda_2(L)|\Sigma| \leq -4\pi^2.
\]  
Moreover, if equality holds, then \(\Sigma\) is the Clifford torus, up to conformal transformations of \(\mathbb{S}^3\).
\end{corollary}

\section{Spectral characterization of surfaces in $\S^1(r)\times\S^2(s)$}\label{torus}

In this section, we present the proof of Theorem \ref{theorem2}, which provides a sharp upper bound for the second eigenvalue of the Jacobi operator of compact surfaces immersed in product spaces of the form \( \mathbb{S}^1(r) \times \mathbb{S}^2(s) \). To simplify the expressions and computations, we apply a homothety so that the radii satisfy the normalization condition
\[
r^2 + s^2 = 1,
\]
which implies that the ambient manifold \( M = \mathbb{S}^1(r) \times \mathbb{S}^2(s) \) is isometrically embedded in the unit sphere \( \mathbb{S}^4 \subset \mathbb{R}^5 \). Under this normalization, the statement of Theorem \ref{theorem2} becomes:

\begin{theorem}
Let \(\Sigma\) be a closed, orientable surface of positive genus immersed in $M = \mathbb{S}^1(r) \times \mathbb{S}^2(\sqrt{1-r^2})$. If $r \geq 1/\sqrt{2}$, then the second eigenvalue \(\lambda_2(L)\) of the Jacobi operator of \(\Sigma\) in \(M\),
\[
L = -\Delta - (|\sigma|^2 + \operatorname{Ric}_M(N,N)),
\]
satisfies
\[
\lambda_2(L) \leq 0.
\]
Furthermore, if \(\lambda_2(L) = 0\), then $r=1/\sqrt{2}$, and \(\Sigma\) is congruent to the totally geodesic torus $\mathbb{S}^1(1/\sqrt{2})\times \mathbb{S}^1(1/\sqrt{2})$ in \(\mathbb{S}^1(1/\sqrt{2})\times \mathbb{S}^2(1/\sqrt{2})\).
\end{theorem}
\begin{proof}
Let $\varphi > 0$ be a first eigenfunction of $L$. By Li and Yau's argument \cite{LiYau}, there exists $y \in \mathbb{B}^5$ such that
\[
\int_\Sigma \varphi (F_y \circ \x) \, dv = 0,
\]
i.e.,
\[
\int_\Sigma \varphi \psi_i \, dv = 0, \quad i = 1, \ldots, 5,
\]
where $\x: \Sigma \to M \subset \mathbb{S}^4 \subset \mathbb{R}^5$ is the position vector and $\psi = F_y \circ \x$ has coordinate functions $\psi_1, \dots, \psi_5$. Using $\psi_1, \dots, \psi_5$ as test functions for $\lambda_2(L)$ and applying the Gauss equation, we obtain:
\begin{align*}
\lambda_2(L)|\Sigma| &\leq \int_\Sigma |\nabla \psi|^2 \, dv - \int_\Sigma \left( |\sigma|^2 + \operatorname{Ric}_M(N,N) \right) \, dv \\
&= \int_\Sigma |\nabla \psi|^2 \, dv - \frac{1}{2} \int_\Sigma \left( R_M + 4|\vec{H}|^2 + |\sigma|^2 - 2K \right) \, dv,
\end{align*}
where $R_M = \frac{2}{1 - r^2}$ is the scalar curvature of $M = \mathbb{S}^1(r) \times \mathbb{S}^2(\sqrt{1 - r^2})$.

On the other hand, the Gauss equation for $\Sigma \subset \mathbb{S}^4$ gives:
\[
|\tau|^2 = 2 + 4|\vec{H}_\tau|^2 - 2K,
\]
where $\vec{H}_\tau$ is the mean curvature vector of $\Sigma$ in $\mathbb{S}^4$. From Proposition~\ref{prop:bound.psff}, we also have:
\[
|\tau|^2 = |\rho|_\Sigma^2 + |\sigma|^2 \leq \frac{2r^2}{1 - r^2} + |\sigma|^2,
\]
which implies:
\[
|\sigma|^2 \geq 2 + 4|\vec{H}_\tau|^2 - 2K - \frac{2r^2}{1 - r^2}.
\]

Combining the inequalities, we deduce:
\[
\frac{1}{2} \left( R_M + 4|\vec{H}|^2 + |\sigma|^2 - 2K \right) \geq 2 + 2|\vec{H}|^2 + 2|\vec{H}_\tau|^2 - 2K.
\]

Therefore,
\begin{align*}
\lambda_2(L)|\Sigma| 
&\leq \int_\Sigma |\nabla \psi|^2 \, dv - \int_\Sigma \left( 2 + 2|\vec{H}|^2 + 2|\vec{H}_\tau|^2 - 2K \right) \, dv.
\end{align*}

Now, applying Lemma~\ref{lem:area.conformal}:
\[
\int_\Sigma |\nabla \psi|^2 \, dv = 2|\Sigma_y| \leq 2 \int_\Sigma \left( 1 + |\vec{H}_\tau|^2 \right) \, dv,
\]
we get:
\[
\lambda_2(L)|\Sigma| \leq -2 \int_\Sigma |\vec{H}|^2 \, dv + 4\pi \chi(\Sigma) \leq 0.
\]

This proves that $\lambda_2(L) \leq 0$.

Suppose now that $\lambda_2(L) = 0$. Then all inequalities above must be equalities. In particular:
\begin{enumerate}
\item[(i)] $\vec{H} = 0$, so $\Sigma$ is minimal in $M$;
\item[(ii)] $\chi(\Sigma) = 0$, hence $\Sigma$ is a torus;
\item[(iii)] $|\rho|_\Sigma^2 = \dfrac{2r^2}{1 - r^2}$.
\end{enumerate}

By Proposition~\ref{prop:bound.psff}, equality in (iii) implies $r = 1/\sqrt{2}$ and that $\Sigma$ is flat and totally geodesic in $M$. Therefore, $\Sigma$ is a finite covering of the totally geodesic torus $\mathbb{S}^1(1/\sqrt{2}) \times \mathbb{S}^1(1/\sqrt{2}) \subset \mathbb{S}^1(1/\sqrt{2}) \times \mathbb{S}^2(1/\sqrt{2})$. Since $L = \Delta + 2$ and $\lambda_2(L) = 0$, it follows that $\Sigma$ must coincide with this torus.
\end{proof}

\section{Examples of flat tori in $\mathbb{S}^1(r) \times \mathbb{S}^2(s)$, with $r \leq s$}\label{exa}

In this final section, we provide explicit examples that illustrate the necessity of the hypotheses in Theorem \ref{theorem2}, especially the condition \( r \geq s \). For this purpose, we compute the second Jacobi eigenvalue for a family of standard flat tori immersed in the product space \( \mathbb{S}^1(r) \times \mathbb{S}^2(s) \).

Let \( \mathbb{T}^2_{r,t} = \mathbb{S}^1(r) \times \mathbb{S}^1(t) \subset \mathbb{R}^4 \) be the flat torus given by
\[
\mathbb{T}_{r,t}^2 = \{(x_1, y_1, x_2, y_2) \in \mathbb{R}^4 \mid x_1^2 + y_1^2 = r^2, \, x_2^2 + y_2^2 = t^2\}.
\]
Fix \( h \in \mathbb{R} \), and define an immersion into \( \mathbb{S}^1(r) \times \mathbb{S}^2(s) \) by
\[
\phi(x_1, y_1, x_2, y_2) = (x_1, y_1, x_2, y_2, h),
\]
where \( s^2 = t^2 + h^2 \). The image is a flat torus embedded in \( M^3 = \mathbb{S}^1(r) \times \mathbb{S}^2(s) \).

A straightforward computation shows that the principal curvatures of this torus are
\[
\kappa_1 = 0, \quad \kappa_2 = \frac{h}{st}.
\]
The eigenvalues of the Laplacian on \( \mathbb{T}_{r,t}^2 \) are given by
\[
\lambda_{m,n} = \frac{m^2}{r^2} + \frac{n^2}{t^2}, \quad m,n \in \mathbb{Z},
\]
so the first nonzero eigenvalue is
\[
\lambda_2(-\Delta) = \min\left\{\frac{1}{r^2}, \frac{1}{t^2}\right\}.
\]

Since the Jacobi operator is
\[
L = -\Delta - \left( \frac{1}{s^2} + \frac{h^2}{s^2 t^2} \right) = -\Delta - \frac{1}{t^2},
\]
we obtain
\[
\lambda_2(L) = \lambda_2(-\Delta) - \frac{1}{t^2} = \min\left\{\frac{1}{r^2}, \frac{1}{t^2}\right\} - \frac{1}{t^2}.
\]

We now consider two cases:

\paragraph{\textbf{Case 1: \( r \geq t \)}}  
In this case, \( \lambda_2(-\Delta) = \frac{1}{r^2} \leq \frac{1}{t^2} \), and therefore:
\[
\lambda_2(L) = \frac{1}{r^2} - \frac{1}{t^2} \leq 0.
\]
This is consistent with the upper bound in Theorem \ref{theorem2}.

\paragraph{\textbf{Case 2: \( r \leq t \)}}  
Here, \( \lambda_2(-\Delta) = \frac{1}{t^2} \), and hence:
\[
\lambda_2(L) = \frac{1}{t^2} - \frac{1}{t^2} = 0.
\]
In this case, although the second eigenvalue is zero, the torus is \emph{not} totally geodesic (unless \( h = 0 \)). Indeed, since \( s^2 = t^2 + h^2 \), we find that \( r \leq t < s \), and hence \( r < s \).

These examples highlight the sharpness of Theorem \ref{theorem2} and its dependence on the geometric constraint \( r \geq s \).

\section*{Acknowledgments}
We would like to thank Lucas Ambrozio for his interest in this work and for his valuable comments and suggestions.
    The authors acknowledge financial support from the Brazilian National Council for Scientific and Technological Development (CNPq) and the Fundação de Amparo à Pesquisa do Estado de Alagoas (FAPEAL). Specifically, M.B. received funding from CNPq (Grants: 308440/2021-8, 405468/2021-0, and 402463/2023-9); M.C. received funding from CNPq (Grants: 405468/2021-0 and 11136/2023-0) and FAPEAL (Grant: 60030.0000000323/2023); A.B. received funding from CNPq (Grants: 405468 /2021-0 and 309867/2023-1) and FAPEAL (Grant: 60030.0000002254/2022); and I.N. received funding from CNPq (Grant: 402463/2023-9).



\bibliographystyle{amsplain}
\bibliography{bibliography.bib}

\end{document}